\documentclass[12pt]{amsart}
\usepackage[utf8]{inputenc}
\usepackage[T1]{fontenc}
\usepackage{amsfonts}
\usepackage[leqno]{amsmath}
\usepackage{amssymb, comment, float}
\usepackage[dvipsnames]{xcolor}
\usepackage[foot]{amsaddr}
\usepackage[shortlabels]{enumitem}
\usepackage{mathdots}
\usepackage[a4paper, footskip=0.5in,  headheight = 0.5in, top=1.25in, bottom=1.25in,  right=1in,  left=1in]{geometry}
\usepackage{hyperref}
\hypersetup{
colorlinks,
linkcolor=black,
urlcolor=black,
citecolor=black,}
\usepackage[center]{caption}
\usepackage{eufrak}
\usepackage{marvosym}     
\usepackage{latexsym}
\usepackage[nodayofweek]{datetime}
\usepackage{tikz}
\usepackage{subfig}
\usepackage{thm-restate}

\newtheorem{theorem}{Theorem}[section]
\newtheorem{lemma}[theorem]{Lemma}
\newtheorem{conjecture}[theorem]{Conjecture}

\DeclareMathOperator{\tw}{tw}
\DeclareMathOperator{\ta}{{\mathsf{tree}-\alpha}}

\def\dd{\hbox{-}}   

\newcommand{\mca}{\mathcal}
\newcommand{\poi}{\mathbb{N}} 

\newcounter{tbox}
\newcommand{\sta}[1]{\medskip\medskip\refstepcounter{tbox}\noindent{\parbox{\textwidth}{(\thetbox) \emph{#1}}}\vspace*{0.3cm}}
\newcommand{\mylongtitle}[1]{%
  \ifodd\value{page}%
    \protect\parbox{0.97\linewidth}{#1}\hfill%
  \else%
    \hfill\protect\parbox{0.97\linewidth}{#1}%
  \fi%
}

\title[Tree independence number III.]{Tree independence number\\
III. Thetas, prisms and stars}

\author{Maria Chudnovsky$^{\mathsection}$}
\thanks{$^{\mathsection}$ Princeton University, Princeton, NJ, USA. Supported by NSF grant  DMS-2348219  and by AFOSR grant FA9550-22-1-0083.}

\author{Sepehr Hajebi$^{\dagger, \ddagger}$}

\thanks{$^{\dagger}$ Department of Combinatorics and Optimization, University of Waterloo, Waterloo, Ontario, Canada.}

\thanks{$^{\ddagger}$ Corresponding Author (\href{mailto:shajebi@uwaterloo.ca}{\texttt{shajebi@uwaterloo.ca}}).}

\author{Nicolas Trotignon$^{\star}$}

\thanks{$^{\star}$ CNRS, ENS de Lyon, Université Lyon 1, LIP UMR 5668, 69342 Lyon Cedex 07, France. Partially supported by the French National
Research Agency under research grant ANR DIGRAPHS ANR-19-CE48-0013-01
and the LABEX MILYON (ANR-10-LABX-0070) of Université de Lyon, within
the program Investissements d’Avenir (ANR-11-IDEX-0007) operated by
the French National Research Agency (ANR), and H2020-MSCA-RISE project CoSP- GA No. 823748}

\date{\today}

\begin{document}
 \maketitle
 
\begin{abstract}
We prove that for every $t\in \poi$ there exists $\tau=\tau(t)\in \poi$ such that every (theta,~prism, $K_{1,t}$)-free graph has tree independence number at most $\tau$ (where we allow ``prisms'' to have one path of length zero).
 \end{abstract}


\section{Introduction}\label{sec:intro}
Graphs in this paper have finite and non-empty vertex sets, no loops and no parallel edges. The set of all positive integers is denoted by $\poi$, and for every $n\in \poi$, we write $[n]$ for the set of all positive integers no greater than $n$. 

Let $G=(V(G),E(G))$ be a graph.  A {\em clique} in $G$ is a set of pairwise adjacent vertices. A {\em stable} or {\em independent} set in $G$ is a set of vertices no two of which are adjacent.
The maximum cardinality of a stable set is denoted by $\alpha(G)$, and the maximum cardinality of a clique in $G$ is denoted by $\omega(G)$. For a graph $H$ we say that $G$ {\em contains $H$} if $H$ is
isomorphic to an induced subgraph of $G$. We say that
$G$ is {\em $H$-free} if $G$ does not contain $H$.
For a set $\mathcal{H}$ of graphs, $G$ is
{\em $\mathcal{H}$-free} if $G$ is $H$-free for every $H \in \mathcal{H}$. For a subset $X$ of $V(G)$, we denote by $G[X]$ the induced subgraph of $G$ with vertex set $X$, we often use ``$X$'' to denote both the set $X$ of vertices and the graph $G[X]$.

Let $X \subseteq V(G)$.  We write $N_G(X)$ for the set of all vertices in $G \setminus X$ with at least one neighbor in $X$, and we define $N_G[X]=N_G(X) \cup X$.   When there is no danger of confusion, we omit the subscript ``$G$''. For $Y\subseteq V(G)$, we write $N_Y(X)=N_G(X)\cap Y$ and $N_Y[X]=N_Y(X)\cup X$. When $X = \{v\}$ is a singleton, we write $N_Y(v)$ for $N_Y(\{v\})$ and $N_Y[v]$ for $N_Y[\{v\}]$. 

Let $x\in V(G)$ and let $Y\subseteq V(G)$. We say that \textit{$x$ is complete to $Y$ in $G$} if $N_Y[x]=Y$, and we say that \textit{$x$ is anticomplete to $Y$ in $G$} if $N_G[x]\cap Y=\emptyset$. In particular, if $x\in Y$, then $x$ is neither complete nor anticomplete to $Y$ in $G$. For subsets $X,Y$ of $V(G)$, we say that \textit{$X$ and $Y$ are complete in $G$} if every vertex in $X$ is complete to $Y$ in $G$, and we say that \textit{$X$ and $Y$ are anticomplete in $G$} if every vertex in $X$ is anticomplete to $Y$ in $G$. In particular, if $X$ and $Y$ are either complete or anticomplete in $G$, then $X\cap Y=\emptyset$. 

For a graph $G = (V(G),E(G))$, a \emph{tree decomposition} $(T, \beta)$ of $G$ consists of a tree $T$ and a map $\beta: V(T) \to 2^{V(G)}$ with the following properties: 
\begin{itemize}
    \item For every $v \in V(G)$, there exists $t \in V(T)$ with $v \in \beta(t)$. 
    
    \item For every $v_1v_2 \in E(G)$, there exists $t \in V(T)$ with $v_1, v_2 \in \beta(t)$.
    
    \item $T[\{t \in V(T) \mid v \in \beta(t)\}]$ is connected for all $v\in V(G)$.
\end{itemize}

The \textit{treewidth} of $G$, denoted $\tw(G)$, is the smallest integer $w\in \poi$ such that $G$ admits a tree decomposition $(T,\beta)$ with $|\beta(t)|\leq w+1$ for all $t\in V(T)$. The \textit{tree independence number} of $G$, denoted $\ta(G)$, is the smallest integer $s\in \poi$ such that $G$ admits a tree decomposition $(T,\beta)$ with $\alpha(G[\beta(t)])\leq s$ for all $t\in V(T)$.
\medskip

Both the treewidth and the tree independence number are of great interest in structural and algorithmic graph theory (see \cite{ti1, tw15, ti2, dfgkm, dms2} for detailed discussions). They are also related quantitatively because, by Ramsey's theorem \cite{multiramsey}, graphs of bounded clique number and bounded tree independence number have bounded treewidth (see also Lemma 3.2 in \cite{dms2}). Dallard, Milani\v{c}, and \v{S}torgel \cite{dms2} conjectured that the converse is also true in \textit{hereditary} classes of graphs (meaning classes which are closed under taking induced subgraphs). Let us say that a graph class $\mca{G}$ is \textit{$(\tw,\omega)$-bounded} if there is a function $f:\poi\rightarrow \poi$ such that every graph $G\in \mca{G}$ satisfies $\tw(G)\leq f(\omega(G))$.

\begin{conjecture}[Dallard, Milani\v{c}, and \v{S}torgel \cite{dms2}]\label{conj:false}
    For every hereditary class $\mca{G}$ which is $(\tw,\omega)$-bounded, there exists $\tau=\tau(\mca{G})\in \poi$ such that $\ta(G)\leq \tau$ for all $G\in \mca{G}$.
\end{conjecture}

Conjecture~\ref{conj:false} was recently refuted \cite{lwnew} by two of the authors of this paper. It is still natural to ask: which $(\tw,\omega)$-bounded hereditary classes have bounded tree independence number? So far, the list of hereditary classes known to be of bounded tree independence number is not very long (see \cite{ti1, dms3, dms2} for a few). More hereditary classes are known to be $(\tw,\omega)$-bounded. The reasons for the existence of the bound are often highly non-trivial, and it is not known whether the corresponding class has bounded tree independence number. A notable instance is the class of all (theta, prism)-free graphs excluding a fixed forest \cite{tw8}, which we will focus on in this paper.
\medskip

Let us first give a few definitions. Let $P$ be a graph which is a path. Then we write, for $k\in \poi$,
$P=p_1\dd \cdots \dd p_k$
to mean $V(P)=\{p_1, \ldots, p_k\}$, and for all $i,j\in [k]$, the vertices  $p_i$ and $p_j$ are adjacent in $P$ if and only if $|i-j| = 1$. We call the vertices $p_1$ and $p_k$ the \textit{ends of $P$}, and we say that $P$ is a \textit{path from $p_1$ to $p_k$} or a \textit{path between $p_1$ and $p_k$}. We refer to $V(P)\setminus \{p_1,p_k\}$ as the \emph{interior of $P$} and denote it by $P^*$. The \textit{length} of a path is its number of edges. Given a graph $G$, by a \textit{path in $G$} we mean an induced subgraph of $G$ which is a path. Similarly, for $t\in \poi\setminus \{1,2\}$, given a $t$-vertex graph $C$ which is a cycle, we write $C=c_1\dd \cdots \dd c_t\dd c_1$
to mean $V(C)=\{c_1, \ldots, c_t\}$, and for all $i,j\in [t]$, the vertices  $c_i$ and $c_j$ are adjacent in $C$ if and only if $|i-j|\in \{1,t-1\}$.  The \textit{length} of a cycle is its number of edges (which is the same as its number of vertices). For a graph $G$, a \textit{hole} in $G$ is an induced subgraph of $G$ which is a cycle of length at least four.
\begin{figure}[t!]
    \centering
    \includegraphics[scale=0.8]{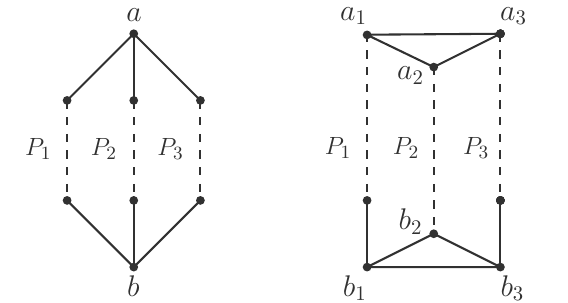}
    \caption{A theta (left) and a prism (right). Dashed lines represent paths of arbitrary (possibly zero) length.}
    \label{fig:thetaprism}
\end{figure}

A {\em theta} is a graph $\Theta$ consisting of two non-adjacent vertices $a, b$, called the \textit{ends of $\Theta$}, and three pairwise internally disjoint paths $P_1, P_2, P_3$ of length at least two in $\Theta$ from $a$ to $b$, called the \textit{paths of $\Theta$}, such that $P_1^*, P_2^*, P_3^*$ are pairwise anticomplete in $\Theta$ (see Figure~\ref{fig:thetaprism}). A {\em prism} is a graph $\Pi$ consisting of two triangles $\{a_1,a_2,a_3\}, \{b_1,b_2,b_3\}$ called the \textit{triangles of $\Pi$}, and three pairwise disjoint paths $P_1,P_2,P_3$ in $\Pi$, called the \textit{paths of $\Pi$}, such that for each $i\in \{1,2,3\}$, $P_i$ has ends $a_i,b_i$, for all distinct $i,j\in \{1,2,3\}$, $a_ia_j$ and $b_ib_j$ are the only edges of $\Pi$ with an end in $P_i$ and an end in $P_j$, and  for every $i \neq  j \in \{1,2,3\}$ $P_i \cup P_j$ is a hole (see Figure~\ref{fig:thetaprism}). If follows that if
$P_2$ has length zero, then each of $P_1, P_3$ has length at least two. We remark that the last condition is non-standard; the paths of a prism are usually of non-zero length, and a prism with a length-zero path is sometimes called a ``line-wheel.'' For a graph $G$, a \textit{theta in $G$} is an induced subgraph of $G$ which is a theta and a \textit{prism in $G$} is an induced subgraph of $G$ which is a prism.
\medskip

The following was proved in \cite{tw8} to show that the local structure of the so-called ``layered wheels'' \cite{layered-wheels} is realized in all theta-free graphs of large treewidth. It also characterizes all forests, and remains true when only  the usual ``prisms'' (with no length-zero path) are excluded: 

\begin{theorem}[Abrishami, Alecu, Chudnovsky, Hajebi, Spirkl \cite{tw8}]\label{thm:tw8main}
   Let $F$ be a graph. Then the class of all (theta, prism, $F$)-free graphs is $(\tw,\omega)$-bounded if and only if $F$ is a forest. 
\end{theorem}

We propose the following strengthening (again, this may be true even with the usual ``prisms'' excluded):

\begin{conjecture}\label{conj:TPForestconj}
  For every forest $F$, there is a constant $\tau=\tau(F)\in \poi$ such that for every (theta, prism, $F$)-free graph $G$, we have $\ta(G)\leq \tau$.  
\end{conjecture}
 As far as we know, Conjecture~\ref{conj:TPForestconj} remains open even for paths. But our main result settles the case of stars. For every $t\in \poi$, let $\mca{C}_t$ be the class of all (theta, prism, $K_{1,t}$)-free graphs. We prove that:

\begin{theorem}\label{thm:main}
For every $t\in \poi$, there is a constant $f_{\ref{thm:main}}=f_{\ref{thm:main}}(t)\in \poi$ such that every graph $G\in \mca{C}_t$ satisfies $\ta(G)\leq f_{\ref{thm:main}}$.
\end{theorem}




\section{Outline of the main proof}\label{sec:outline}

Like several earlier results \cite{ti1,ti2,tw15} coauthored by the first two authors of this work, the proof of Theorem~\ref{thm:main} deals with ``balanced separators.'' Let $G$ be a graph and let $w:V(G)\rightarrow \mathbb{R}^{\geq 0}$. For every $X\subseteq V(G)$, we write $w(X)=\sum_{v\in X}w(v)$. We say that that $w$ is a \textit{normal weight function on $G$} if $w(V(G))=1$. Given a graph $G$ and a weight function $w$ on $G$, a subset $X$ of $V(G)$ is called a \textit{$w$-balanced separator} if for every component $D$ of $G\setminus X$, we have $w(D)\leq 1/2$. The main step in the proof of Theorem~\ref{thm:main} is the following:

\begin{theorem}\label{thm:balancedsep}
For every $t\in \poi$, there is a constant $f_{\ref{thm:balancedsep}}=f_{\ref{thm:balancedsep}}(t)\in \poi$ with the following property.
Let $G \in \mca{C}_t$ and let $w$ be a normal weight function on $G$.
Then there exists $Y \subseteq V(G)$ such that $|Y|\leq f_{\ref{thm:balancedsep}}$ and $N[Y]$ is a $w$-balanced separator in $G$.
\end{theorem}
As shown below, Theorem~\ref{thm:main} follows by  combining Theorem~\ref{thm:balancedsep} and the following (this is not a difficult result; see \cite{ti2} for a proof):
\begin{lemma}[Chudnovsky, Gartland, Hajebi, Lokshtanov and Spirkl; see Lemma 7.1 in  \cite{ti2}]\label{lemma:bs-to-treealpha}
Let $s\in \poi$ and let $G$ be a graph. If for every normal weight function $w$ on $G$, there is a
$w$-balanced separator $X_w$ in $G$ with $\alpha(X_w) \leq s$, then we have  $\ta(G)\leq 5s$.
\end{lemma}

\begin{proof}[Proof of Theorem~\ref{thm:main} assuming Theorem~\ref{thm:balancedsep}]
Let $c=f_{\ref{thm:balancedsep}}(t)$. We prove that $f_{\ref{thm:main}}(t)=5ct$ satisfies the theorem. Let $w$ be a normal weight function on $G$. By Theorem~\ref{thm:balancedsep}, there exists $Y \subseteq V(G)$ such that $|Y|\leq c$ and $X_w=N[Y]$ is a $w$-balanced separator in $G$.  Assume that there is a stable set $S$ in $X_w$ with $|S|>ct$. Since $S\subseteq N[Y]$, it follows that there is a vertex $y\in Y$ with $|N[y]\cap S|\geq t$. But now $G$ contains $K_{1,t}$, a contradiction. We deduce that $\alpha(X_w)\leq ct$. Hence, by Lemma~\ref{lemma:bs-to-treealpha}, we have $\ta(G)\leq 5ct=f_{\ref{thm:main}}(t)$. This completes the proof of Theorem~\ref{thm:main}.
\end{proof}
It remains to prove Theorem~\ref{thm:balancedsep}. The idea of the proof is the following.  In \cite{tw15} a technique was developed to prove that separators satisfying the conclusion of Theorem~\ref{thm:balancedsep} exist. It consists of showing that the graph class in question satisfies two properties: being ``amiable'' and being ``amicable.'' Here we use the same technique.  To prove that a graph class is amiable, one needs to analyze the structure of connected subgraphs containing neighbors of a given set of vertices. To prove that a graph is amicable, it is necessary to show that certain carefully chosen pairs of vertices can  be separated by well-structured separators. Most of the remainder of the paper is devoted to these two tasks.
Section~\ref{sec:wheel} and Section~\ref{sec:strip} contain structural results asserting the existence of separators that will be used to establish amicability. Section~\ref{sec:connectifier} contains definitions and previously known results related to amiability.  Section~\ref{sec:amiable} contains the proof of the fact that the class $\mathcal{C}_t$ is amiable.
Section~\ref{sec:amicable} uses the results of Sections~\ref{sec:wheel} and \ref{sec:strip} to deduce that 
$\mathcal{C}_t$ is amicable, and to complete the proof of  Theorem~\ref{thm:balancedsep}. 


\section{Breaking a wheel}\label{sec:wheel}

A \emph{wheel} in a graph $G$ is a pair $W=(H,c)$ where $H$ is a hole in $G$ and $c\in G\setminus H$ has at least three neighbors in $H$. We also use $W$ to denote the vertex set $H\cup \{c\}\subseteq V(G)$. A \emph{sector} of the wheel $(H,c)$ is a path of non-zero length in $H$
  whose ends are adjacent to $c$ and whose
 internal vertices are not. A wheel is \emph{special} if it has exactly three sectors, one sector
 has length one and the other two (called the \emph{long} sectors) have
 length at least two (see Figure~\ref{fig:wheel} -- A special wheel is sometimes referred to as a ``short pyramid.'')

 For a wheel $W=(H,c)$ in a graph $G$, we define the set $Z(W)\subseteq H\cup \{c\}$ as follows (see Figure~\ref{fig:wheel}). If $W$ is non-special, then $Z(W)=N_H[c]$. Now assume that $W$ is special. Let $ab$ be
  the sector of length one of $W$ and let $d$ be the neighbor of $c$ in
  $H\setminus \{a, b\}$. Then we define $Z(W)=\{a,b,c\}\cup N_H[d]$. 
  
  \begin{figure}[t!]
     \centering
     \includegraphics[scale=0.8]{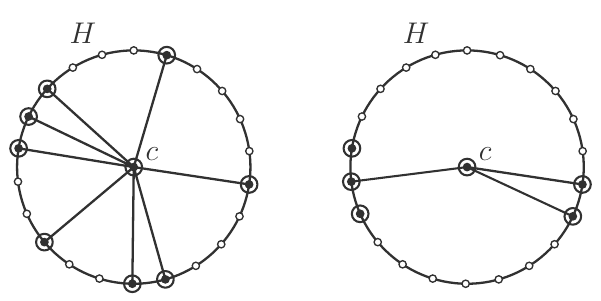}
     \caption{A non-special wheel $W$ (left) and a special wheel $W$ (right). Circled nodes represent the vertices in $Z(W)$.}
     \label{fig:wheel}
 \end{figure}

 Let $G$ be a graph. By a \textit{separation} in $G$ we mean a triple $(L,M,R)$ of pairwise disjoint subsets of $V(G)$ with $L\cup M\cup R=V(G)$, such that neither $L$ nor $R$ is empty and $L$ and $R$ are anticomplete in $G$. Let $x,y\in V(G)$ be distinct. We say that a set $M\subseteq V(G)\setminus \{x,y\}$ \textit{separates $x$ and $y$ in $G$} if there exists a separation $(L,M,R)$ in $G$ with $x\in L$ and $y\in R$. Also, for disjoint sets $X,Y\subseteq V(G)$, we say that a set $M\subseteq V(G)\setminus (X\cup Y)$ \textit{separates $X$ and $Y$} if there exists a separation $(L,M,R)$ in $G$ with $X\subseteq L$ and $Y\subseteq R$. If $X=\{x\}$, we say that \textit{$M$ separates $x$ and $Y$} to mean $M$ separates $X$ and $Y$.
\medskip

We have two results in this section; one for the non-special wheels and one for  special wheels:

\begin{theorem}\label{thm:wheel1}
  Let $G$ be a (theta, prism)-free graph, let $W=(H, c)$ be a non-special
  wheel in $G$ such that $H$ has length at least seven. Let $a, b\in G\setminus N[Z(W)]$ belong to (the interiors of)
  distinct sectors of $W$.  Then $N[Z(W)]$ separates $a$ and $b$ in $G$.
\end{theorem}

\begin{proof}
Let $S=N_H(c)$ and let $T = N[c] \cup (N[S] \setminus H)$. Then $T\subseteq N[Z(W)]$, and so it suffices to show that $T$ separates $a$ and $b$ (note that $a,b\notin T$). We begin with the following:

\sta{\label{st:claim1} Assume that some vertex $v \in G \setminus (W \cup T)$ has
either a unique neighbor or two non-adjacent neighbors in some sector
$P = p \dd \cdots\dd  p'$ of $W$. Let $b'$ be the neighbor of $p$ in
$W\setminus P$ and  $b''$ be the neighbor of $p'$ in
$W\setminus P$. Then $N_H(v) \subseteq P \cup \{b', b''\}$.}

Otherwise, $v$ has a neighbor
$d \in H \setminus (P \cup \{b', b''\})$.  Also, $c$ has a neighbor
$d' \in H \setminus (P \cup \{b', b''\})$, as otherwise $W$ would
be a prism or a special wheel.  We choose $d$ and $d'$ such that the path $Q$ in $H\setminus P$ from $d$ to $d'$ is minimal. If $v$ has a unique neighbor $a$ in $P$, then $P\cup Q\cup \{c,v\}$ is a theta in $G$ with ends $a$ and $c$, a contradiction. Also, if $v$
has two non-adjacent neighbors in $P$, then $P\cup Q\cup \{v,c\}$ contains a theta with ends $c$ and $v$. This proves \eqref{st:claim1}.

\sta{\label{st:claim2} For every $v \in G \setminus (W \cup T)$, there
exists a sector $P$ of $W$ such that $N_H(v) \subseteq P$.}

Suppose there exists a sector $P = p \dd \cdots\dd  p'$ such that $v$ has two
non-adjacent neighbors in $P$.  Then, by \eqref{st:claim1}, we may assume up to
symmetry that $v$ is adjacent to the neighbor $b$ of $p$ in
$H\setminus P$.  By \eqref{st:claim1}, $b$ is the unique neighbor of $v$ in some
sector $Q$ of $W$.  So the fact that $v$ has at least two neighbors in
$P$ contradicts \eqref{st:claim1} applied to $v$ and $Q$.

Suppose there exists a sector $P = p \dd \cdots\dd  p'$ such that $v$ has a
unique neighbor $a$ in $P$.  By \eqref{st:claim1}, we may assume that
$N_H(v) = \{a, b', b''\}$ where $b'$ is the neighbor of $p$ in $W\setminus P$ and $b''$ is the neighbor of $p'$ in $W\setminus P$ (because $N_H(v) = \{a, b'\}$ or
$N_H(v) = \{a, b''\}$ would imply that $v$ and $H$ form a theta).  Let
$Q=p \dd \cdots\dd  q$ be the sector of $W$ that contains $b'$.  By \eqref{st:claim1}
applied to $v$ and $Q$, we have $ap\in E(G)$ and $b''q \in E(G)$.  So,
$b''$ is the unique neighbor of $v$ in the sector $R = p'\dd \cdots\dd  q$ of $W$.
By \eqref{st:claim1} applied to $v$ and $R$, we have $ap'\in E(G)$ and
$b'q \in E(G)$.  So $H$ has length six, a contradiction.

We proved that for every sector $P$ of $W$, either $v$ has no neighbors
in $P$, or $v$ has two neighbors in $P$, and those
neighbors are adjacent.  We may therefore assume that $v$ has
neighbors in at least three distinct sectors of $W$, because if $v$
has neighbors in exactly two of them, then $H\cup \{v\}$ would be a
prism.  So, suppose that $P= p\dd \cdots\dd  p'$, $Q = q\dd \cdots\dd  q'$ and
$R = r\dd \cdots\dd  r'$ are three distinct sectors of $W$, and $v$ is adjacent
to $x, x' \in P$, to $y, y' \in Q$ and to $z, z' \in R$.
Suppose up to symmetry that $p$, $x$, $x'$, $p'$, $q$, $y$, $y'$, $q'$,
$r$, $z$, $z'$ and $r'$ appear in this order along $H$.  Then there is a theta in $G$ with ends $c,v$ and paths 
$v\dd x\dd P\dd p\dd c$, $v\dd y\dd Q\dd q\dd c$ and $v\dd z\dd R\dd r\dd c$, a
contradiction.  This proves \eqref{st:claim2}. 
\medskip
 
To conclude the proof, suppose for a contradiction that the interiors
of two distinct sectors of $W$ are contained in the same connected component of
$G \setminus T$.  Then there exists a
path $Y = v\dd \cdots\dd  w$ in $G\setminus T$
and two sectors $P = p\dd \cdots\dd  p'$ and $Q = q\dd \cdots\dd  q'$ of $W$ such that
$v$ has neighbors in  $P^*$ and $w$ has neighbors in $Q^*$.  By \eqref{st:claim2}, $v$ is anticomplete to $W\setminus P$
and $w$ is anticomplete to $W \setminus Q$ (in particular, $Y$ has
length at least one).  By choosing such a path $Y$ to be minimal, we deduce
that $Y^*$ is anticomplete to $H$.

Suppose that $v$ has a unique neighbor, or two distinct and
non-adjacent neighbors in $P$.  Next, assume that $w$ has a neighbor $d$ in $H$ that
is distinct from $b'$ and $b''$ where $b'$ is the neighbor of $p$ in $W\setminus P$ and $b''$ is the neighbor of $p'$ in $W\setminus P$, then let $d'$ be a neighbor of
$c$ in $H \setminus (P \cup \{b', b''\})$ ($d'$ exists for otherwise,
$W$ would be a prism or a special wheel).  We choose $d$ and $d'$
such that the path $R$ in
$H\setminus P$ from $d$ to $d'$ is minimal. We now see that if $v$ has a unique neighbor $a$ in
$P$, then  $P \cup Y \cup R \cup \{c\}$ contains a theta with ends $a$ and $c$, a
contradiction. 
Also, if $v$ has two distinct non-adjacent neighbors in
$P$, then $P\cup Y\cup R\cup \{c\}$ contains a theta with ends $c$ and $v$.  So,
$w$ has only two possible neighbors in $H$, namely, $b'$ and $b''$.  Due to symmetry, we may assume that $b'w\in E(G)$ (so $b''w\notin E(G)$). It follows that $b'$  is non-adjacent to $c$.  If $v$
has a unique neighbor in $P$, then $H\cup Y$ is a theta in $G$, so $v$
has a neighbor in $P$ that is non-adjacent to $p$.  In particular,
there exists a path $R'$ from $v$ to $p'$ in $P\cup \{v\}$ that
contains no neighbor of $p$.  It follows that $R'\cup Q\cup Y\cup \{c\}$
is a theta in $G$ with ends $b'$ and $c$.

We deduce that $v$ has exactly two neighbors in $P$, and those neighbors are adjacent.  By the same argument, we can prove that $w$ has
exactly two neighbors in $P$ that are adjacent.  But now $H\cup Y$ is a prism in $G$, a contradiction. This completes the proof of Theorem~\ref{thm:wheel1}.
\end{proof}

\begin{theorem}\label{thm:wheel2}
  Let $G$ be a (theta, prism)-free graph and let $W= (H, c)$ be a special
  wheel in $G$ whose long sectors have lengths at least three. Let $a'', b''\in G\setminus N[Z(W)]$ belong to (the interiors of)
  distinct sectors of $W$.  Then $N[Z(W)]$ separates $a''$ and $b''$ in $G$.  
\end{theorem}

\begin{proof}
  Let $ab$ be
  the sector of length one of $W$ and let $d$ be the neighbor of $c$ in
  $H\setminus \{a, b\}$. Let $a'$ be the neighbor of $d$ in the long sector of $W$ containing $a$ and let $b'$ be the neighbors of $d$ in
  the long sector of $W$ containing $b$. Then $Z(W)=\{a, a', b, b', c, d\}$. Let
  $P$ be the path in $H\setminus d$ from $a$ to $a'$ and let $Q$ be the path of $H\setminus d$ from $b$ to $b'$. Assume, without loss of generality, that $a''\in P^*\setminus N[Z(W)]$ and let $b''\in Q^*\setminus N[Z(W)]$.

  Let $T= N[c] \cup (N[\{a, b, a', b', d\}] \setminus H)$. Then $T\subseteq N[Z(W)]$, and so it suffices to show that $T$ separates $a''$ and $b''$ (note that $a'',b''\notin T$).  Suppose not. Then there exists a path $Y = v\dd \cdots\dd  w$ in
  $G\setminus T$ such that $v$ has neighbors in $P^*$,
  $w$ has neighbors in  $Q^*$, $Y\setminus v$ is
  anticomplete to $W\setminus P$ and $Y\setminus w$ is
  anticomplete to $W\setminus Q$ (note that possibly $v=w$).

  Let $x$  be the neighbor of $v$ in $P$ closest to $a$ along $P$ and let $x'$  be the neighbor of $v$ in $P$ closest to $a'$ along $P$.  Let $y$ be the neighbor of
  $w$ in $Q$ closest to $b$ along $Q$ and let $y'$ be the neighbor of
  $w$ in $Q$ closest to $b'$ along $Q$.

  If $x=x'$, then there is a theta in $G$ with ends $x$ and $d$ and paths  $x\dd P\dd a'\dd d$, $x\dd P\dd a\dd c\dd d$ and $x\dd v\dd Y\dd w\dd y'\dd Q\dd b'\dd d$.  So, $x\neq x'$, and symmetrically we have 
  $y\neq y'$.  If $xx'\notin E(G)$, then there is a theta in $G$ with ends $v$ and $d$ and paths $v\dd x'\dd P\dd a'\dd d$, $v\dd x\dd P\dd a\dd c\dd d$ and
  $v\dd Y\dd w\dd y'\dd Q\dd b'\dd d$. So, $xx'\in E(G)$, and
  symmetrically we can prove that $yy' \in E(G)$.  But now $H\cup Y$
  is a prism in $G$, a contradiction. This completes the proof of Theorem~\ref{thm:wheel2}.
\end{proof}

\section{Breaking a pyramid}\label{sec:strip}

A {\em pyramid} is a graph $\Sigma$ consisting of a vertex $a$, a triangle $\{b_1, b_2, b_3\}$ disjoint from $a$ and three paths $P_1,P_2, P_3$ in $\Sigma$ of length at least two, such that for each $i\in [3]$, the ends of $P_i$ are $a$ and $b_i$, and for all distinct $i,j\in [3]$, the sets $V(P_i) \setminus \{a\}$ and $V(P_j) \setminus \{a\}$ are disjoint,  $b_ib_j$ is the only edge of $G$ with an end in  $V(P_i) \setminus \{a\}$ and an end in $V(P_j) \setminus \{a\}$, and for every $i \neq j \in \{1,2,3\}$ $P_i \cup P_j$ is a hole (the assumption that $P_1,P_2, P_3$ have length at least two is non-standard; usually, one of the paths is allowed to have length $1$, and our definition above would refer to a ``long'' pyramid.)

We say that $a$ is the {\em apex} of $\Sigma$, the triangle $\{b_1,b_2,b_3\}$ is the {\em base} of $\Sigma$, and $P_1,P_2,P_3$ are the \textit{paths} of $\Sigma$.  We also define $Z(\Sigma)=N_{\Sigma}[a]\cup \{b_1,b_2,b_3\}$ (so we have $|Z(\Sigma)|=7$). For a graph $G$, by a \textit{pyramid in $G$} we mean an induced subgraph of $G$ which is a pyramid (see Figure~\ref{fig:pyramid}).

\begin{figure}[t!]
    \centering
    \includegraphics[scale=0.8]{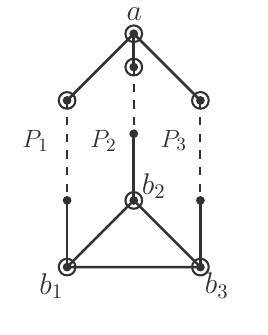}
    \caption{A pyramid $\Sigma$. Dashed lines represent paths of arbitrary (possibly zero) length, and circled nodes represent the vertices in $Z(\Sigma)$.}
    \label{fig:pyramid}
\end{figure}
\medskip

The main result of this section, Theorem~\ref{thm:pyramidmain} below, follows from much more general results of \cite{tw8}.  However, there is  
also a short and self-contained proof, which we include here:

\begin{theorem}\label{thm:pyramidmain}
 Let $G$ be a (theta, prism)-free graph and let $\Sigma$ be a pyramid in $G$ with apex $a$, base $\{b_1,b_2,b_3\}$ and paths $P_1,P_2$ and $P_3$ as in the definition. Let $u,v\in G\setminus N[Z(\Sigma)]$ belong to distinct paths of $\Sigma$. Then $N[Z(\Sigma)]$ separates $u$ and $v$ in $G$.
\end{theorem}

\begin{proof}
Suppose not. Then there exist $u,v\in G\setminus N[Z(\Sigma)]$, belonging to distinct paths of $\Sigma$, such that $N[Z(\Sigma)]$ does not separate $u$ and $v$ in $G$.  It follows that for distinct $i,j\in [3]$, there exists a path $Q =x\dd \cdots \dd y$ in $G\setminus (\Sigma \cup N[Z(\Sigma)])$ such that $x$ has a neighbor in $P^*_i$ and $y$ has a neighbor in $P^*_j$. We choose $i,j\in [3]$ and $Q$ subject to the minimality of $Q$. By symmetry, we may assume that $i=1$ and $j=2$.

From the minimality of $Q$ and the fact that $Q\subseteq V(G)\setminus (\Sigma \cup N[Z(\Sigma)])$, it follows that:

\begin{itemize}
    \item $N_{P_1}(x)\subseteq P_1\setminus Z(\Sigma)$, and $Q\setminus x$ and $P_1$ are anticomplete in $G$.
    \item $N_{P_2}(y)\subseteq P_2\setminus Z(\Sigma)$, and $Q\setminus y$ and $P_2$ are  anticomplete in $G$. 
\end{itemize}

Now, if some vertex of $Q$ has a neighbor in $P_3$, then by the minimality of $Q$, we must have $x=y$. In particular, $x$ has neighbors in $P_1$, $P_2$ and $P_3$. Since $a$ and $x$ are not adjacent in $G$ (for otherwise there is a theta in $G$), it follows that the three paths in $G$ from $a$ to $x$ with interiors in $P_1$, $P_2$ and $P_3$ form a theta in $G$ with ends $a$ and $x$, a contradiction.  We deduce that $Q$ and $P_3$ are anticomplete in $G$.  

Let $x'$ be the neighbor of $x$ in $P_1$ closest to $a$ along $P_1$ and let $x''$ be the neighbor of $x$ in $P_1$ closest to $b_1$ along $P_1$. Similarly, let $y'$ be the neighbor of $y$ in $P_2$ closest to $a$ along $P_2$ and let $y''$ be the neighbor of $y$ in $P_2$ closest to $b_2$ along $P_2$. Recall that $x',x''\in P_1\setminus Z(\Sigma)$ and $y',y''\in P_2\setminus Z(\Sigma)$. If $x'=x''$, then there is a theta in $G$ with ends $a,x'$ and paths $a\dd P_1\dd x'$, $a\dd P_2\dd y'\dd y\dd Q\dd x\dd x'$ and $a\dd P_3\dd b_3\dd b_1\dd P_1\dd x'$.  Also, if $x'$ and $x''$ are distinct and adjacent in $G$, then there is a prism in $G$ with triangles $x''xx'$ and $b_1b_2b_3$ and paths $x''\dd P_1\dd b_1$, $x\dd Q\dd y\dd y''\dd P_2\dd b_2$ and $x'\dd P_1\dd a\dd P_3\dd b_3$.  Hence, we have $x'\neq x''$ and $x'x''\notin E(G)$. But now there is a theta in $G$ with ends $a,x$ and paths $a\dd P_1\dd x'\dd x$, $a\dd P_2\dd y'\dd y\dd Q\dd x$ and $a\dd P_3\dd b_3\dd b_1\dd P_1\dd x''\dd x$, a contradiction. This completes the proof of Theorem~\ref{thm:pyramidmain}.
\end{proof}

\section{Alignments and Connectifiers}\label{sec:connectifier}

This section covers a number of definitions and a result from \cite{tw15}, which we will  use in the  proof of Theorem~\ref{thm:balancedsep}.

Let $G$ be a graph, let $P$ be a path in $G$ and let $X\subseteq V(G)\setminus P$.  We say that $(P,X)$ is an
{\em alignment} if every vertex of $X$  has at least one neighbor in $P$ and one may write $P=p_1 \dd \cdots \dd p_n$ and
$X=\{x_1, \ldots, x_k\}$ for $k,n\in \poi$ such that there exist
$1\leq i_1\leq j_1<i_2\leq j_2 < \cdots< i_{k}\leq j_k\leq n$ where 
$N_P(x_l) \subseteq p_{i_l} \dd P \dd p_{j_{l}}$ for every $l \in [k]$. 
This is a little different from the definition in \cite{tw15}, but the difference is not substantial, 
and using this definition is more convenient for us here.
In this case, we say that $x_1, \ldots, x_k$ is {\em the order on $X$ given by the
  alignment $(P,X)$}.
An alignment $(P,X)$ is {\em wide}  if 
each of $x_1, \ldots, x_{k}$ has two non-adjacent neighbors in  $P$,
{\em spiky} if  each of $x_1, \ldots, x_{k}$ has a unique neighbor in $P$ and
{\em triangular} if 
each of $x_1, \ldots, x_{k}$ has exactly two neighbors in $P$  and those neighbors are
adjacent. An alignment is {\em consistent} if it is
wide, spiky or triangular. See Figure~\ref{fig:alignment}.
\begin{figure}
    \centering
    \includegraphics[scale=0.8]{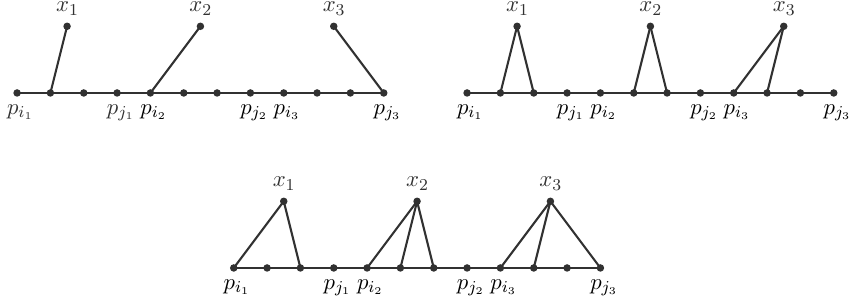}
    \caption{A consistent alignment which is spiky (top left), triangular (top right) and wide (bottom).}
    \label{fig:alignment}
\end{figure}

By a {\em caterpillar} we mean a tree $C$ with
  maximum degree three such that no two branch vertices in $C$ are adjacent, and such that there exists a path $P$ in $C$ containing all
  branch vertices of $C$.
We call a minimal such path $P$ the {\em spine} of $C$. 
  (We note that our definition of a ``caterpillar'' is
  non-standard in multiple ways.) By a \textit{subdivided star} we mean a graph isomorphic to
  a subdivision of the complete bipartite graph $K_{1,\delta}$ for some $\delta\geq 3$.
  In other words, a subdivided star is a tree with exactly one branch vertex,
  which we call its \textit{root}. For a graph $H$, a vertex $v$ of $H$ is
  said to be \textit{simplicial} if $N_H(v)$ is a clique. We denote by
  $\mca{Z}(H)$ the set of all simplicial vertices of $H$. Note that for
  every tree $T$, $\mca{Z}(T)$ is the set of all leaves of $T$. An edge
  $e$ of a tree $T$ is said to be a \textit{leaf-edge} of $T$ if $e$ is
  incident with a leaf of $T$. It follows that if $H$ is the line graph of a
  tree $T$, then $\mca{Z}(H)$ is the set of all vertices in $H$
  corresponding to the leaf-edges of $T$.

 Let $H$ be a graph that is either a caterpillar, or the line graph of a caterpillar,  or a subdivided star with root $r$,  or the line graph of  a subdivided star with root $r$.
We define an induced subgraph of $H$, denoted by $P(H)$, which we will use
throughout the paper.
If $H$ is a path (possibly of length zero), then let $P(H)=H$.
If $H$ is a caterpillar, then let $P(H)$ be the spine of $H$.
  If $H$ is the line graph of a caterpillar $C$, then let $P(H)$ be the path in $H$
  consisting of the vertices of $H$ that correspond to the edges of the spine of $C$. If $H$ is a subdivided star  with root $r$, then let $P(H)=\{r\}$.
  If $H$ is the line graph of a subdivided star $S$  with root $r$, let
  $P(H)$ be the clique of $H$ consisting of the vertices of $H$ that
  correspond to the edges of $S$ incident with $r$.
  The {\em legs} of $H$ are the components of $H \setminus P(H)$. Let $G$ be a graph and let $H$ be an induced subgraph of $G$ that is either a caterpillar, or the line graph of a caterpillar, or a subdivided star or the line graph of a subdivided star. Let $X \subseteq V(G) \setminus H$ such that every vertex of $X$ has a unique neighbor in $H$ and $N_H(X)=\mathcal{Z}(H)$ (see Figure~\ref{fig:connectifier}). We call $(H,X)$ a \textit{connectifier}. Also,
  if $H$ is a single vertex and $X \subseteq N(H)$,
we call $(H,X)$ a {\em  connectifier} as well. 
  We say that the connectifier $(H,X)$ is \textit{concentrated} if $H$ is a subdivided star or the line graph of a subdivided star or a singleton.

Let $(H,X)$ be a connectifier in $G$ which is not concentrated. So $H$ is a caterpillar or the line graph of a caterpillar. Let $S$ be the set of vertices of $H \setminus P(H)$ that have neighbors in
$P(H)$. Then $(P(H),S)$ is an alignment.
Let $s_1, \ldots, s_k$ be the corresponding order
on $S$ given by $(P(H),S)$. Now, order the vertices of $X$ as
$x_1, \ldots, x_k$ where for every $i \in [k]$, the vertex $x_i$ has a neighbor in the leg of $H$ containing
$s_i$. We say that $x_1, \ldots, x_k$ is {\em the order on $X$ given by $(H,X)$}.

The following was proved in \cite{tw15}:

\begin{theorem} [Chudnovsky, Gartland, Hajebi, Lokshtanov and Spirkl; Theorem 5.2 in  \cite{tw15}]
\label{thm:connectifier2general}
For every integer $h\in \poi$, there is a constant $f_{\ref{thm:connectifier2general}}=f_{\ref{thm:connectifier2general}}(h)\in \poi$ with the following property. Let $G$ be a connected graph.
Let $S\subseteq V(G)$ such that  $|S|\geq f_{\ref{thm:connectifier2general}}$, the graph $G \setminus S$ is connected  and every vertex of $S$ has a neighbor in $G\setminus S$. Then there exists $S' \subseteq S$ with $|S'|=h$ as well as an induced subgraph $H$ of $G \setminus S$ for which one of the following holds.
  \begin{itemize}
  \item $(H,S')$ is a connectifier, or
  \item $H$ is a path and every vertex in $S'$ has a neighbor in $H$.
\end{itemize}
\end{theorem}
\begin{figure}[t!]
    \centering
    \includegraphics[scale=0.8]{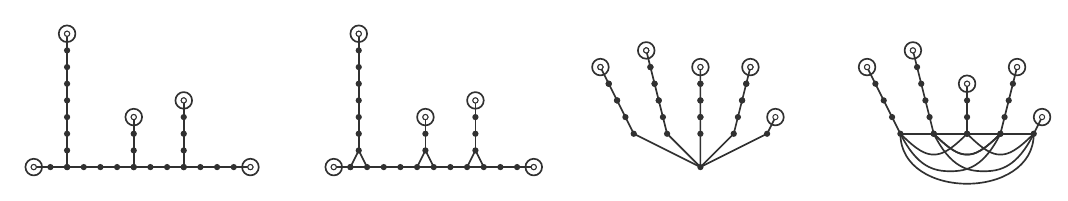}
    \caption{Examples of a connectifier. Circled nodes represent the vertices in $X$.}
    \label{fig:connectifier}
\end{figure}

\section{Amiability}\label{sec:amiable}
 The two notions of ``amiability'' and ``amicability,'' first introduced in \cite{tw15}, are at the heart of the proof of Theorem~\ref{thm:balancedsep}. We deal with the former in this section and leave the latter for the next one.

Let $s\in \poi$ and let $G$ be a graph. An \textit{$s$-trisection in $G$} is a separation $(D_1,Y,D_2)$ in $G$ such that the following hold.
\begin{itemize}
   \item $Y$ is a stable set with $|Y|=s$.
   \item $D_1$ and $D_2$ are components of $G\setminus Y$ with $N(D_1)=N(D_2)=Y$.
  \item $D_1$ is a path and for every $y \in Y$ there exists $d_y\in D_1$ such
    that $N_Y(d_y)=\{y\}$. 
  \end{itemize}
(The reader may notice that we will never use the second condition in the third bullet point. It was however necessary in \cite{tw15}, so we keep it for easier cross-referencing.)

We say that a graph class $\mca{G}$ is \textit{amiable} if there is a function $\sigma:\poi\rightarrow\poi$ with the following property. Let $x\in \poi$, let $G \in \mca{G}$  and let $(D_1,Y,D_2)$ be a $\sigma(x)$-trisection in $G$. Then there exist $H\subseteq D_2$ and $X\subseteq Y$ with $|X|=x$ such that the following hold.
\begin{itemize}
\item $(D_1,X)$ is a consistent alignment.
\item $(H,X)$ is either a connectifier or a consistent alignment.
\item If $(H,X)$ is not a concentrated connectifier, then
  the orders given on $X$ by $(D_1,X)$ and by $(H,X)$ are the same.
\end{itemize}
In this case, we say that  \textit{$H$ and $X$ are given by amiability}. The main result of this section is the following:
\begin{theorem} 
\label{thm:amiable}
For every $t\in \poi$, the class $\mca{C}_t$ is amiable. Moreover, with notation as in the definition of amiability, if $(H,X)$ is a connectifier, then we have $|H|>1$.
\end{theorem}

In order to prove Theorem~\ref{thm:amiable}, first we prove the following lemma:

\begin{lemma}\label{lem:intervalgraph}
Let $d,s\in \poi$, let $G$ be a theta-free graph and let $Y$ be a stable set in $G$ of cardinality $3s(d+1)$. Let $P$ be a path in $G\setminus Y$ such that every vertex in $Y$ has a neighbor in $P$, and each vertex of $P$ has fewer than $d$ neighbors in $Y$. Assume that for every two vertices $y,y'\in Y$, there is a path $R$ in $G$ from $y$ to $y'$ such that $P$ and $R^*$ are disjoint and anticomplete in $G$. Then there is an $s$-subset $S$ of $Y$ such that $(P,S)$ is a consistent alignment.
\end{lemma}
\begin{proof}
  For every vertex $y\in Y$, let $P_y$ be the (unique) path in $P$ with the property that $y$ is complete to the ends of $P_y$ and anticomplete to $P\setminus P_y$. Let $I$ be the graph with $V(I)=Y$ such that two distinct vertices $y,y'\in Y$ are adjacent in $I$ if and only if $P_y\cap P_{y'}\neq \emptyset$. Then $I$ is an interval graph and so $I$ is perfect \cite{Golumbic}. Since $|V(I)|=3s(d+1)$, it follows that $I$ contains either a clique of cardinality $d+1$ or a stable set of cardinality $3s$. 
    
  Assume that $I$ contains a clique of cardinality $d+1$. Then there exists $C\subseteq Y$ with $|C|=d+1$ and $p\in P$ such that $p \in P_y$ for every $y \in C$. Since $p\in P$ has fewer than $d$ neighbors in $C\subseteq Y$, it follows that there are at least two vertices $y,y'\in C\setminus N(p)$. Since $p\in P_y\cap P_{y'}$, it follows that $P\setminus \{p\}$ has two components, and each of $y$ and $y'$ has a neighbor in each component of $P\setminus \{p\}$.  It follows that there are two paths $P_1$ and $P_2$ from $y$ to $y'$ with disjoint and anticomplete interiors contained in $P$. On the other hand, there is a path $R$ in $G$ from $y$ to $y'$ such that $P$ and $R^*$ are disjoint and anticomplete in $G$. It follows that $P_1,P_2$ and $R$ are pairwise internally disjoint and anticomplete. But now there is a theta in $G$ with ends $y,y'$ and paths $P_1, P_2, R$, a contradiction.

We deduce that $I$ contains a stable set $S'$ of cardinality $3s$. From the definition of $I$, it follows that $(P,S')$ is an alignment. Hence, since every vertex in $S'$ has one, two adjacent, or at least two non-adjacent neighbors in $P$, there exists $S\subseteq S'\subseteq Y$ with $|S|=s$ such that $(P,S)$ is a consistent alignment. This completes the proof of Lemma~\ref{lem:intervalgraph}.
\end{proof}

\begin{proof}[Proof of Theorem~\ref{thm:amiable}]
 For every $x\in \poi$, let 
$$s=f_{\ref{thm:connectifier2general}}(3x^2(t+1))$$
and let
 $$\sigma(x)=3s(t+1).$$
  We will show that $\mca{C}_t$ is amiable with respect to $\sigma:\poi\rightarrow \poi$ as defined above. Let $x\in \poi$, let $G \in \mca{C}_t$  and let $(D_1,Y,D_2)$ be a $\sigma(x)$-trisection in $G$. Then $Y$ is a stable set of cardinality $3s(t+1)$, $D_1$ is a path in $G\setminus Y$ and every vertex in $Y$ has a neighbor in $D_1$. Moreover, since $G$ is $K_{1,t}$-free, no vertex in $D_1$ has $t$ or more neighbors in $Y$, and since $N(D_2)=Y$, it follows that for every two vertices $y,y'\in Y$, there is a path $R$ in $G$ from $y$ to $y'$ with $R^*\subseteq D_2$, and so $D_1$ and $R^*$ are disjoint and anticomplete in $G$. By Lemma~\ref{lem:intervalgraph}, there exists $S\subseteq Y$ with $|S|=s$ such that $(D_1,S)$ is a consistent alignment.

 Now, we show that there exists $H\subseteq D_2$ as well as an $x$-subset $X$ of $S\subseteq Y$ such that $H$ and $X$ satisfy the definition of amiability. Since $D_2$ is connected and every vertex in $S\subseteq Y$ has a neighbor in $D_2$, it follows that $D_2\cup S$ is connected too. Since $|S|=s=f_{\ref{thm:connectifier2general}}(3x^2(t+1))$, it follows from Theorem~\ref{thm:connectifier2general} that  there exists $S' \subseteq S$ with $|S'|=3x^2(t+1)$ and an induced subgraph $H_2$ of $D_2$ for which one of the following holds:
  \begin{itemize}
  \item $(H_2,S')$ is a connectifier.
  \item $H_2$ is a path and every vertex of $S'$ has a neighbor in $H_2$.
\end{itemize}
 First, assume that $(H_2,S')$ is a concentrated connectifier. Then, since $|S'|\geq t$ and $G$ is $K_{1,t}$-free, it follows that $|H_2|>1$. Now, since $|S'|\geq x$, we may choose a concentrated connectifier $(H,X)$ where $X$ is an $x$-subset of $S'\subseteq S\subseteq Y$ and $H$ is an induced subgraph $H_2\subseteq D_2$ with $|H|>1$. In particular, $H$ and $X$ satisfy the definition of amiability.

 Next, assume that $(H_2,S')$ is a connectifier which is not concentrated. Consider the orders on $S'$ given by $(D_1, S')$ and by $(H_2,S')$. Since $|S'|\geq x^2$, it follows from the Erd\H{o}s-Szekeres theorem \cite{ES} that there is an $x$-subset $X$ of $S'\subseteq S\subseteq Y$ as well as an induced subgraph $H$ of $H_2\subseteq D_2$
 such that:
 \begin{itemize}
 \item $(D_1,X)$ is a consistent alignment (because $(D_1,S)$ is);
\item $(H,X)$ is a connectifier which is not concentrated; and
\item The orders given on $X$ by $(D_1,X)$ and by $(H,X)$ are the same.
 \end{itemize}
It follows that $H$ and $X$ satisfy the definition of amiability.

Finally, assume that $H_2$ is a path and every vertex in $S'$ has a neighbor in $H_2$. Let $H=H_2$. Recall that $(D_1,S')$ is an alignment. In particular, $S'$ is a stable set of cardinality $3x^2(t+1)$, and since $G$ is $K_{1,t}$-free, no vertex in $H_2$ has $t$ or more neighbors in $S'$. Also, for every two vertices $y,y'\in S$, there is a path $R$ in $G$ from $y$ to $y'$ such that $R^*\subseteq D_1$, and so $H$ and $R^*$ are disjoint and anticomplete in $G$.  By Lemma~\ref{lem:intervalgraph}, there exists $S''\subseteq S'\subseteq S$ with $|S''|=x^2$ such that $(H,S'')$ is a consistent alignment. Consider the order on $S''$ given by $(D_1, S'')$ and by $(H,S'')$. Since $|S''|=x^2$, it follows from the Erd\H{o}s-Szekeres theorem \cite{ES} that there is an $x$-subset $X$ of $S''\subseteq S'\subseteq S\subseteq Y$ such that
 such that:
 \begin{itemize}
 \item $(D_1,X)$ is a consistent alignment (because $(D_1,S)$ is);
\item $(H,X)$ is a consistent alignment (because $(H,S'')$ is); and
\item The orders given on $X$ by $(D_1,X)$ and by $(H,X)$ are the same.
 \end{itemize}
So $H$ and $X$ satisfy the definition of amiability. This completes the proof of Theorem~\ref{thm:amiable}
\end{proof}

\section{Amicability}\label{sec:amicable}
Here we complete the proof of Theorem~\ref{thm:balancedsep}, beginning with the following definition.

Let $m\in \poi$ and let $\mca{G}$ be a graph class. We say that $\mca{G}$ is \textit{$m$-amicable} if $\mca{G}$ is amiable and the following holds. Let $\sigma:\poi\rightarrow \poi$ be as in the definition of amiability for $\mca{G}$. Let $G\in \mca{G}$ and let $(D_1,Y,D_2)$ be a $\sigma(7)$-trisection in $G$. Let $X=\{x_1,\ldots, x_7\}\subseteq Y$ be given by amiability such that $x_1, \ldots, x_7$ is the order on $X$ given by $(D_1,X)$. Let $D_1=d_1 \dd \cdots \dd d_k$ such that traversing $D_1$ from $d_1$ to $d_k$, the first vertex in $D_1$ with a neighbor in $X$ is a neighbor of $x_1$. Let $i\in [k]$ be maximum such that $x_1$ is adjacent to $d_i$ and let $j\in [k]$ be minimum such that $x_7$ is adjacent to $d_j$. Then there exists a subset $Z$ of  $D_2 \cup \{d_{k}:i+2\leq k\leq j-2\} \cup  \{x_4\}$
    with $|Z| \leq m$  such that $N[Z]$ separates  $d_i$ and $d_j$. It follows that $N[Z]$ separates $d_1 \dd D_1 \dd d_i$ and
    $d_j \dd D_1 \dd d_k$ (see Figure~\ref{fig:amicable}).
    \begin{figure}[t!]
        \centering
        \includegraphics[scale=0.7]{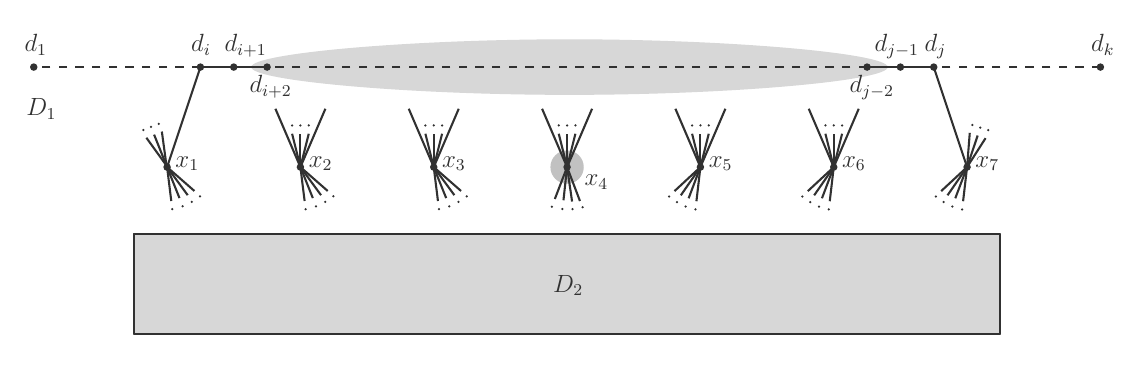}
        \caption{Amicability -- Note that $Z$ is contained in the highlighted set.}
        \label{fig:amicable}
    \end{figure}
    \begin{figure}[t!]
    \centering
    \includegraphics[scale=0.7]{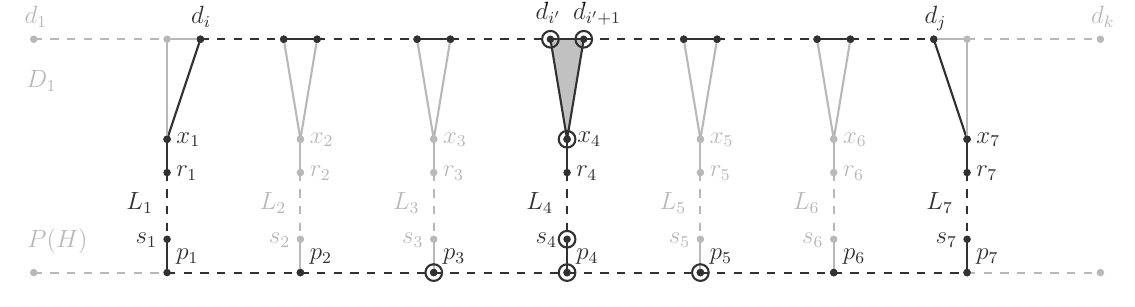}
    \caption{$H$ is a caterpillar. Circled nodes depict the vertices in $Z(\Sigma)$.}
    \label{fig:caterpillarcase}
\end{figure}
    
We prove that:
  
  \begin{theorem}
    \label{thm:amicable}
   For every $t\in \poi$, the class $\mca{C}_t$ is $\max\{2t,7\}$-amicable.
    \end{theorem}

    \begin{proof}
    By Theorem~\ref{thm:amiable}, $\mca{C}_t$ is amiable, and with notation as in the definition of amiability, if $(H,X)$ is a connectifier, then we have $|H|>1$. Let $\sigma:\poi\rightarrow \poi$ be as in the definition of amiability for $\mca{C}_t$. Let $G\in \mca{C}_t$ and let $(D_1,Y,D_2)$ be a $\sigma(7)$-trisection in $G$. Let $X=\{x_1,\ldots, x_7\}\subseteq Y$ be given by amiability such that $x_1, \ldots, x_7$ is the order on $X$ given by the consistent alignment $(D_1,X)$. Let $D_1=d_1 \dd \cdots \dd d_k$ and $i,j\in [k]$ be as in the definition of amicability. Our goal is to show that there exists a subset $Z$ of  $D_2 \cup \{d_{k}:i+2\leq k\leq j-2\} \cup  \{x_4\}$
    with $|Z| \leq \max\{2t,7\}$  such that $N[Z]$ separates  $d_i$ and $d_j$.

Let $i'\in [k]$ be minimum such that $x_4$ is adjacent to $d_{i'}$, let $j'\in [k]$ be maximum such that $x_4$ is adjacent to $d_{j'}$, and let $H$ be the induced subgraph of $D_2$ given by amiability. It follows that
$i+2<i'\leq j'<j-2$, $(H,X)$ is either a connectifier with $|H|>1$ or a consistent alignment, and if $(H,X)$ is not a concentrated connectifier, then $x_1,\ldots,x_7$ is the order on $X$ given by $(H,X)$. 
When $(H,X)$ is a connectifier with $|H|>1$, then for each $l\in [7]$, let $r_l$ be the unique neighbor of $x_l$ in $H$ (so $r_l\in \mca{Z}(H)$) and let $L_l$ be the (unique) shortest path in $H$ from $r_l$ to a vertex $s_l\in N_H[P(H)]$. It follows that $s_l\in H\setminus P(H)$ unless $H$ is the line graph of a subdivided star where not all edges of the star are subdivided, in which case we have $r_l=s_l\in P(H)=\mca{Z}(H)=H$.
\medskip

First, consider the case where  $H$ is a caterpillar. It follows that for each $l\in [7]$, we have $s_l\in H\setminus P(H)$ and $s_l$ has a unique neighbor $p_l\in P(H)$. Since $G$ is theta-free, it follows that $(D_1,X)$ is triangular, and so $j'=i'+1$ (see Figure~\ref{fig:caterpillarcase}). Let $\Sigma$ be the pyramid with apex $p_4$, base $\{d_{i'},x_4,d_{j'}\}$ and paths 
$$P_1=p_4\dd P(H)\dd p_1\dd s_1\dd L_1\dd r_1\dd x_1\dd d_i\dd D_1\dd d_{i'};$$
$$P_2=p_4\dd s_4\dd L_4\dd r_4\dd x_4;$$
$$P_3=p_4\dd P(H)\dd p_7\dd s_7\dd L_7\dd r_7\dd x_7\dd d_j\dd D_1\dd d_{j'}.$$
Then $Z(\Sigma)$ is a $7$-subset of $D_2 \cup \{d_{k}:i+2\leq k\leq j-2\} \cup  \{x_4\}$. Moreover, we have $d_i\in P_1^*\setminus N[Z(\Sigma)]$ and $d_j\in P_3^*\setminus N[Z(\Sigma)]$. Therefore, by Theorem~\ref{thm:pyramidmain}, $N[Z(\Sigma)]$ separates $d_i$ and $d_j$, as desired.

\begin{figure}[t!]
    \centering
    \includegraphics[scale=0.7]{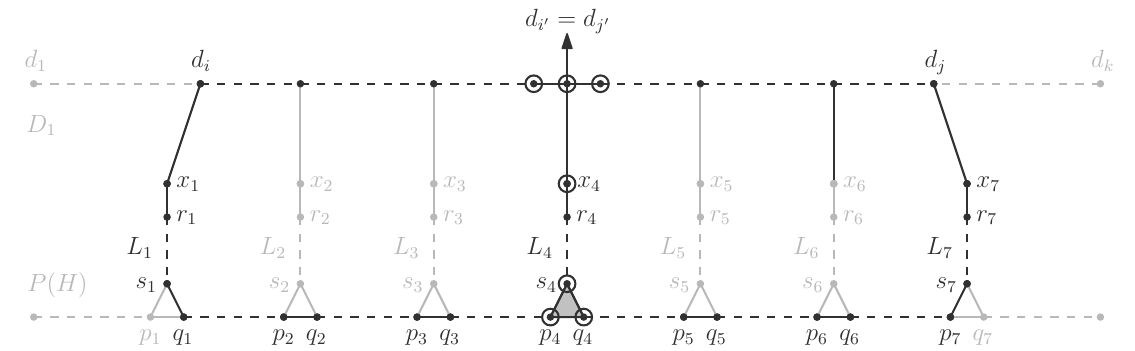}
    \caption{$H$ is the line graph of a caterpillar and $(D_1,X)$ is spiky. Circled nodes represent the vertices in $Z(\Sigma)$.}
    \label{fig:linecaterpillarspikycase}
\end{figure}
\begin{figure}[t!]
    \centering
    \includegraphics[scale=0.7]{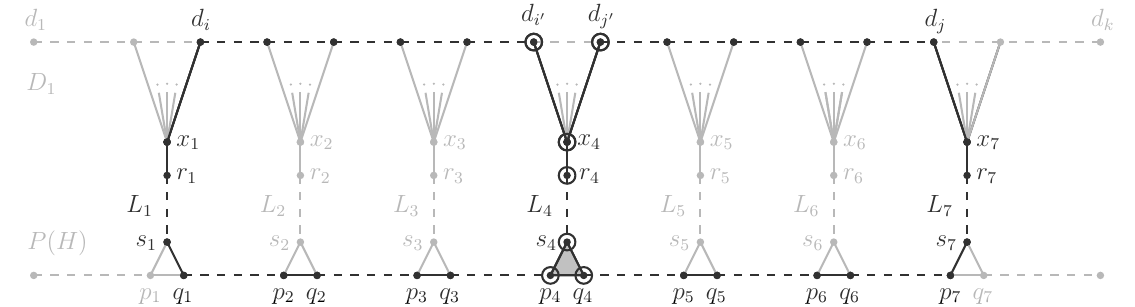}
    \caption{$H$ is the line graph of a caterpillar and $(D_1,X)$ is wide. Circled nodes represent the vertices in $Z(\Sigma)$.}
    \label{fig:linecaterpillarwidecase}
\end{figure}
\medskip

 Second, consider the case where $H$ is the line graph of a caterpillar. It follows that for each $l\in [7]$, we have $s_l\in H\setminus P(H)$ and $s_l$ has exactly two neighbors $p_l,q_l\in P(H)$, where $p_l$ and $q_l$ are adjacent, and the vertices $p_1,q_1,p_2,q_2,\ldots, p_7,q_7$ appear on $P(H)$ in this order. Since $G$ is prism-free, it follows that $(D_1,X)$ is either spiky or wide. Suppose that $(D_1,X)$ is spiky (see Figure~\ref{fig:linecaterpillarspikycase}). Then $i'=j'$. Let $\Sigma$ be the pyramid with apex $d_{i'}=d_{j'}$, base $\{p_4,s_4,q_4\}$ and paths 
$$P_1=d_{i'}\dd D_1\dd d_i\dd x_1\dd r_1\dd L_1\dd s_1\dd q_1\dd P(H)\dd p_4;$$
$$P_2=d_{i'}\dd x_4\dd r_4\dd L_4\dd s_4;$$
$$P_3=d_{i'}\dd D_1\dd d_j\dd x_7\dd r_7\dd L_7\dd s_7\dd p_7\dd P(H)\dd q_4.$$
Then $Z(\Sigma)$ is a $7$-subset of $D_2 \cup \{d_{k}:i+2\leq k\leq j-2\} \cup  \{x_4\}$. Moreover, we have $d_i\in P_1^*\setminus N[Z(\Sigma)]$ and $d_j\in P_3^*\setminus N[Z(\Sigma)]$. So by Theorem~\ref{thm:pyramidmain}, $N[Z(\Sigma)]$ separates $d_i$ and $d_j$. Now assume that  $(D_1,X)$ is wide (see Figure~\ref{fig:linecaterpillarwidecase}). Then $j'-i'>1$. Let $\Sigma$ be the pyramid with apex $x_4$, base $\{p_4,s_4,q_4\}$ and paths 
$$P_1=x_4\dd d_{i'}\dd D_1\dd d_i\dd x_1\dd r_1\dd L_1\dd s_1\dd q_1\dd P(H)\dd p_4;$$
$$P_2=x_4\dd r_4\dd L_4\dd s_4;$$
$$P_3=x_4\dd d_{j'}\dd D_1\dd d_j\dd x_7\dd r_7\dd L_7\dd s_7\dd p_7\dd P(H)\dd q_4.$$
Let $Z=(N(x_4)\cap \Sigma)\cup \{p_4,s_4,q_4\}$. Then $Z(\Sigma)$ is a $7$-subset of $D_2 \cup \{d_{k}:i+2\leq k\leq j-2\} \cup  \{x_4\}$. Also, we have $d_i\in P_1^*\setminus N[Z(\Sigma)]$ and $d_j\in P_3^*\setminus N[Z(\Sigma)]$. So by Theorem~\ref{thm:pyramidmain}, $N[Z(\Sigma)]$ separates $d_i$ and $d_j$, as required.
\begin{figure}[t!]
    \centering
    \includegraphics[scale=0.7]{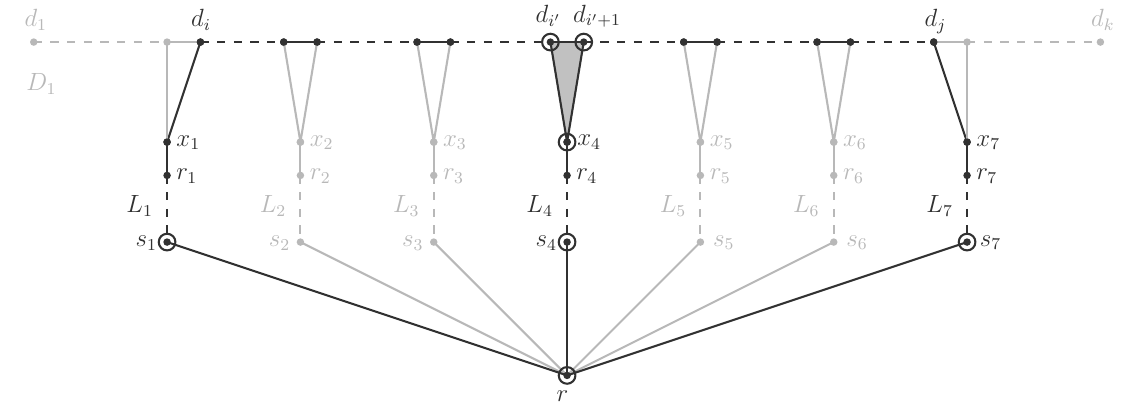}
    \caption{$H$ is a subdivided star. Circled nodes represent the vertices in $Z(\Sigma)$.}
    \label{fig:substarcase}
\end{figure}
\begin{figure}[t!]
    \centering
    \includegraphics[scale=0.7]{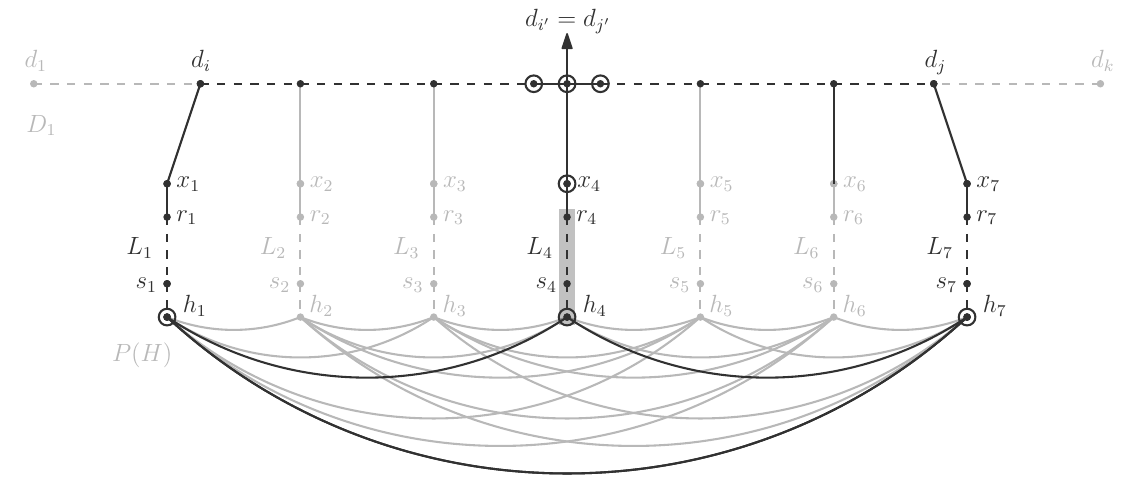}
    \caption{$H$ is the line graph of a subdivided star and $(D_1,X)$ is spiky. Circled nodes represent the vertices in $Z(\Sigma)$, and the highlighted path may be of length zero.}
\label{fig:linesubstarcasespikycase}
\end{figure}
\medskip

Third, consider the case where $H$ is a subdivided star with root $r$. It follows that $P(H)=\{r\}$ and $H \neq \{r\}$ (because $|H|>1$). For each $l\in [7]$, we have $r_l,s_l\in H\setminus P(H)$ and $r_l$ is a leaf of $H$. Since $G$ is theta-free, it follows that $(D_1,X)$ is triangular and so $j'-i'=1$ (see Figure~\ref{fig:substarcase}). Let $\Sigma$ be the pyramid with apex $r$, base $\{d_{i'},x_4,d_{j'}\}$ and paths 
$$P_1=r\dd s_1\dd L_1\dd r_1\dd x_1\dd d_i\dd D_1\dd d_{i'};$$
$$P_2=r\dd s_4\dd L_4\dd r_4\dd x_4;$$
$$P_3=r\dd s_7\dd L_7\dd r_7\dd x_7\dd d_j\dd D_1\dd d_{j'}.$$
Then $Z(\Sigma)$ is a $7$-subset of $D_2 \cup \{d_{k}:i+2\leq k\leq j-2\} \cup  \{x_4\}$. Also, we have $d_i\in P_1^*\setminus N[Z(\Sigma)]$ and $d_j\in P_3^*\setminus N[Z(\Sigma)]$. So it follows from Theorem~\ref{thm:pyramidmain} that $N[Z(\Sigma)]$ separates $d_i$ and $d_j$, as desired.
\begin{figure}[t!]
    \centering
    \includegraphics[scale=0.7]{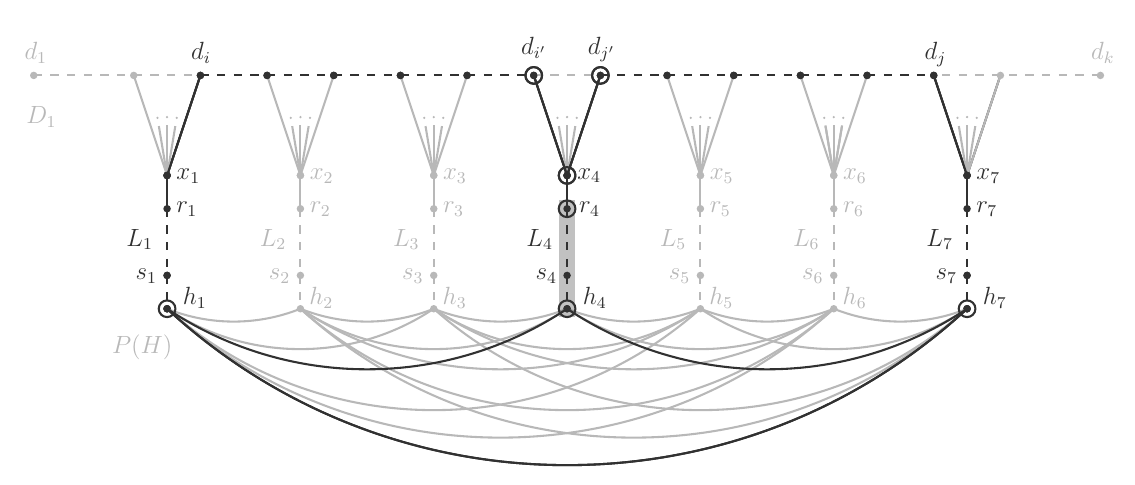}
    \caption{$H$ is the line graph of a subdivided star, $(D_1,X)$ is wide and the vertices $r_4,s_4,h_4$ are not all the same. 
Circled nodes represent the vertices in $Z(\Sigma)$, and the highlighted path has length at least one.}
\label{fig:linesubstarcasewidecase1}
\end{figure}
   \begin{figure}[t!]
    \centering
    \includegraphics[scale=0.7]{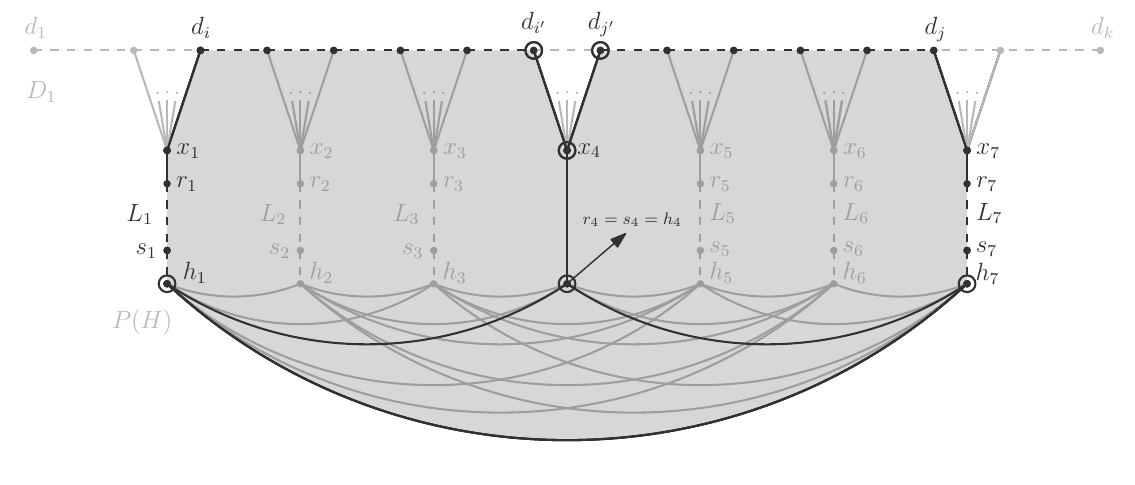}
    \caption{$H$ is the line graph of a subdivided star, $(D_1,X)$ is wide and $r_4=s_4=h_4$. The hole $C$ is highlighted, and circled nodes represent the vertices in $Z(W)$.}
\label{fig:linesubstarcasewidecase2}
\end{figure}
\medskip

Fourth, consider the case where $H$ is the line graph of a subdivided star. It follows that for each $l\in [7]$, either we have $s_l\in P(H)$, in which case we set $h_l=s_l$, or we have $s_l\in H\setminus P(H)$, in which case we choose $h_l$ to be the unique neighbor of $s_l$ in $P(H)$. Since $G$ is prism-free, it follows that $(D_1,X)$ is either spiky or wide. There are now three cases to analyze:
\begin{figure}[t!]
    \centering
    \includegraphics[scale=0.7]{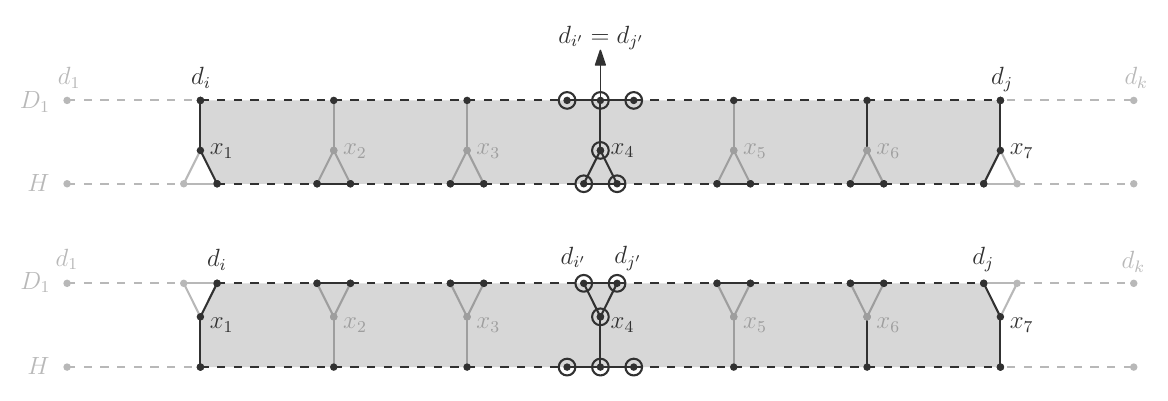}
    \caption{One of $(D_1,X)$ and $(H,X)$ is spiky and the other is triangular. The hole $C$ is highlighted, and circled nodes represent the vertices in $Z(W)$.}
    \label{fig:spikyvstriangularcase}
\end{figure}
\begin{figure}[t!]
    \centering
    \includegraphics[scale=0.7]{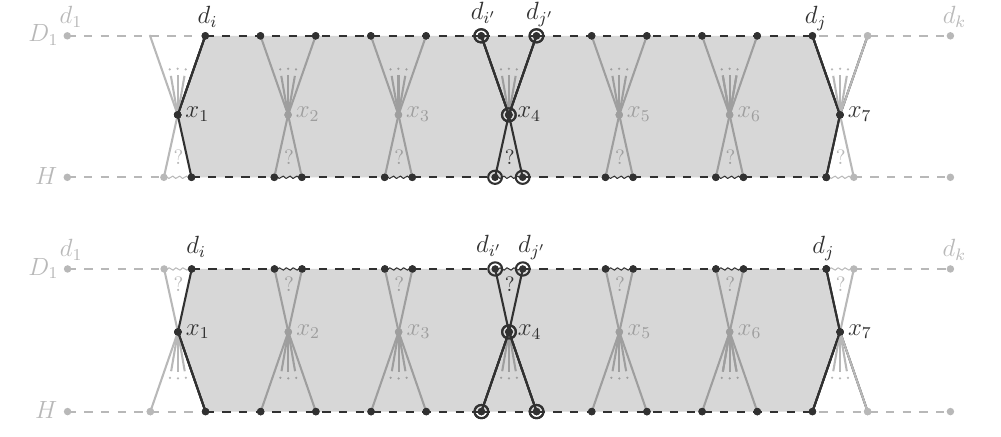}
    \caption{One of $(D_1,X)$ and $(H,X)$ is wide. The hole $C$ is highlighted, and circled nodes represent the vertices in $Z(W)$.}
    \label{fig:widevsnonwidecase}
\end{figure}
\begin{itemize}
\item[\sl Case 1.] Suppose that $(D_1,X)$ is spiky (see Figure~\ref{fig:linesubstarcasespikycase}). Then we have $i'=j'$. Consider the pyramid $\Sigma$ in $G$ with apex $d_{i'}=d_{j'}$, base $\{h_1,h_4,h_7\}$ and paths 
$$P_1=d_{i'}\dd D_1\dd d_i\dd x_1\dd r_1\dd L_1\dd s_1\dd h_1;$$
$$P_2=d_{i'}\dd x_4\dd r_4\dd L_4\dd s_4\dd h_4;$$
$$P_3=d_{i'}\dd D_1\dd d_j\dd x_7\dd r_7\dd L_7\dd s_7\dd h_7.$$
Then $Z(\Sigma)$ is a $7$-subset of $D_2 \cup \{d_{k}:i+2\leq k\leq j-2\} \cup  \{x_4\}$. Moreover, we have $d_i\in P_1^*\setminus N[Z(\Sigma)]$ and $d_j\in P_3^*\setminus N[Z(\Sigma)]$. Thus, by Theorem~\ref{thm:pyramidmain}, $N[Z(\Sigma)]$ separates $d_i$ and $d_j$.

\item[\sl Case 2.]  Suppose that $(D_1,X)$ is wide and the vertices $r_4,s_4,h_4$ are not all the same (see Figure~\ref{fig:linesubstarcasewidecase1}).  Then $j'-i'>1$. Let $\Sigma$ be the pyramid with apex $x_4$, base $\{h_1,h_4,h_7\}$ and paths 
$$P_1=x_4\dd d_{i'}\dd D_1\dd d_i\dd x_1\dd r_1\dd L_1\dd s_1\dd h_1;$$
$$P_2=x_4\dd r_4\dd L_4\dd s_4\dd h_4;$$
$$P_3=x_4\dd d_{j'}\dd D_1\dd d_j\dd x_7\dd r_7\dd L_7\dd s_7\dd h_7.$$
Then $Z(\Sigma)$ is a $7$-subset of $D_2 \cup \{d_{k}:i+2\leq k\leq j-2\} \cup  \{x_4\}$, and we have $d_i\in P_1^*\setminus N[Z(\Sigma)]$ and $d_j\in P_3^*\setminus N[Z(\Sigma)]$. It follows from Theorem~\ref{thm:pyramidmain} that $N[Z(\Sigma)]$ separates $d_i$ and $d_j$.

\item[\sl Case 3.] Suppose that $(D_1,X)$ is wide and  $r_4=s_4=h_4$ (see Figure~\ref{fig:linesubstarcasewidecase2}).  Then $j'-i'>1$. Let 
 $C=x_4\dd d_{i'}\dd D_1\dd d_i\dd x_1\dd r_1\dd L_1\dd s_1\dd h_1\dd h_7\dd s_7\dd L_7\dd r_7\dd x_7\dd d_j\dd D_1\dd d_{j'}\dd x_4.$ 
 Then $C$ is a hole on more than seven vertices and $W=(C,h_4)$ is a special wheel in $G$ where $Z(W)=\{d_{i'},d_{j'},h_1,h_4,h_7,x_4\}$; in particular, $Z(W)$ is a $6$-subset of $D_2 \cup \{d_{k}:i+2\leq k\leq j-2\} \cup  \{x_4\}$. By Theorem~\ref{thm:wheel2}, $N[Z(W)]$ separates $d_i$ and $d_j$.
   \end{itemize}  
\medskip

Finally, assume that $(H,X)$ is a consistent alignment. Recall that $(D_1,X)$ is also a consistent alignment, and that $(D_1,X)$ and $(H,X)$ give the same order $x_1,\ldots,x_7$ on $X$. Let $R$ be the unique path in $G$ from $x_1$ to $x_7$ with $R^*\subseteq H$. Then 
$C=d_i\dd x_1\dd R\dd x_7\dd d_j\dd D_1\dd d_i$
is a hole on more than seven vertices in $G$. Also, since $G$ is (theta, prism)-free, it follows that either one of $(D_1,X)$ and $(H,X)$ is spiky and the other is triangular, or at least one of $(D_1,X)$ and $(H,X)$ is wide. In the former case, $W=(C,x_4)$ is a special wheel (see Figure~\ref{fig:spikyvstriangularcase}). It follows from Theorem~\ref{thm:wheel2} that $Z(W)$ is a $6$-subset of $D_2 \cup \{d_{k}:i+2\leq k\leq j-2\} \cup  \{x_4\}$ such that  $N[Z(W)]$ separates $d_i$ and $d_j$. In the latter case, $W=(C,x_4)$ is a non-special wheel (see Figure~\ref{fig:widevsnonwidecase}). Since $G$ is $K_{1,t}$-free, it follows that $Z(W)=N_C[x_4]\subseteq D_2 \cup \{d_{k}:i+2\leq k\leq j-2\} \cup  \{x_4\}$ has cardinality at most $2t$. Moreover, by Theorem~\ref{thm:wheel1}, $N[Z(W)]$ separates $d_i$ and $d_j$. This completes the proof of Theorem~\ref{thm:amicable}. 
\end{proof}

We also need the following result from \cite{tw15}:

\begin{theorem}[Chudnovsky, Gartland, Hajebi, Lokshtanov, Spirkl \cite{tw15}]\label{thm:balancedseptw15}
For every $m\in \poi$ and every $m$-amicable graph class $\mca{G}$, there is a constant $f_{\ref{thm:balancedseptw15}}=f_{\ref{thm:balancedseptw15}}(\mca{G},m)\in \poi$ with the following property. Let $\mca{G}$ be a graph class which is $m$-amicable. Let $G \in \mca{G}$ and let $w$ be a normal weight function on $G$.
Then there exists $Y \subseteq V(G)$ such that
\begin{itemize}
\item $|Y| \leq f_{\ref{thm:balancedseptw15}}$, and
\item $N[Y]$ is a $w$-balanced separator in $G$.
\end{itemize}
\end{theorem}

Now, defining $f_{\ref{thm:balancedsep}}(t)=f_{\ref{thm:balancedseptw15}}(\mca{C}_t,\max\{2t,7\})$ for every $t\in \poi$, Theorem~\ref{thm:balancedsep} is immediate from Theorems~\ref{thm:amicable} and \ref{thm:balancedseptw15}.


\section{Acknowledgement}
 Part of this work was done when Nicolas Trotignon visited Maria Chudnovsky at Princeton University with the generous support of the H2020-MSCA-RISE project CoSP- GA No. 823748.

\bibliographystyle{abbrv}
\bibliography{ref}  

\end{document}